\documentclass[english]{amsart}

\usepackage{amsmath}
\usepackage{amssymb}
\usepackage{amsthm}
\usepackage{amscd}
\usepackage{babel}
\usepackage[latin1]{inputenc}
\usepackage{graphicx}

\DeclareMathOperator{\lcm}{\rm{lcm}}

\newtheorem{teor}{Theorem}
\newtheorem{cor}{Corollary}
\newtheorem{prop}{Proposition}
\newtheorem{lem}{Lemma}

\theoremstyle{definition}
\newtheorem{defi}{Definition}
\newtheorem{rem}{Remark}
\newtheorem{con}{Conjecture}

\renewcommand{\subjclassname}{AMS \textup{2010} Mathematics Subject
Classification\ }

\author{Jos\'{e} Mar\'{i}a Grau}
\address{Departamento de Matemáticas, Universidad de Oviedo\\ Avda. Calvo Sotelo s/n, 33007 Oviedo, Spain}
\email{grau@uniovi.es}

\author{Antonio M. Oller-Marc\'{e}n}
\address{Centro Universitario de la Defensa de Zaragoza\\ Ctra. Huesca s/n, 50090 Zaragoza, Spain} \email{oller@unizar.es}

\author{Jonathan Sondow}
\address{209 West 97th Street\\New York, NY 10025, USA} \email{jsondow@alumni.princeton.edu}

\title[On the congruence $1^m + 2^m + \dotsb + m^m\equiv n \pmod{m}$ with $n\mid m$]{On the congruence\\$\boldsymbol{1^m + 2^m + \dotsb + m^m\equiv n \pmod{m}}$ with $\boldsymbol{n\mid m}$}

\begin{document}

\begin{abstract}
We show that if the congruence above holds and $n\mid m$, then the quotient $Q:=m/n$ satisfies $\sum_{p\mid Q} \frac{Q}{p}+1 \equiv 0\pmod{Q}$, where $p$ is prime. The only known solutions of the latter congruence are $Q=1$ and the eight known primary pseudoperfect numbers $2,6,42, 1806, 47058, 2214502422, 52495396602,$ and $8490421583559688410706771261086$. Fixing $Q$, we prove that the set of positive integers $n$ satisfying the congruence in the title, with $m=Q n$, is empty in case  $Q=52495396602$, and in the other eight cases has an asymptotic density between bounds in  $(0,1)$ that we provide.
\end{abstract}

\maketitle
\subjclassname{11B99, 11A99, 11A07}


\section{Introduction}

This paper deals with {\em power sums} of the form
$$S_m(k):=1^m+2^m+3^m+\cdots+k^m,$$
where $m,k\in  \mathbb{N}:=\{1,2,3,\dotsc\}$. Power sums were first studied in detail by Jakob Bernoulli (1654-1705), leading him to develop the {\em Bernoulli numbers}, as they are known today. In fact, if we denote by $B_i$ and $B_i(x)$ the $i$-th Bernoulli number and Bernoulli polynomial, then (see, e.g., \cite{BEA})
\begin{equation}\label{int}S_m(m)= \frac{B_{m+1}(m+1)-B_{m+1}}{m+1}.\end{equation}

In particular, much work has been done regarding divisibility properties of power sums \cite{GMO,LEN,SON2}.
%S_m(k)=\frac{1}{m+1}\sum_{i=0}^{m} \binom{m+1}{i} (k+1)^{m+1-i} B_i.

The general Diophantine equation 
$$S_m(x)=y^n$$
was considered by Sch\H{a}ffer \cite{SCH} in 1956. In particular he proved that this equation has infinitely many solutions in positive integers $x$ and $y$ if and only if $(m,n)\in\{(1,2),(3,2),(3,4),(5,2)\}$. Moreover, for $m\geq 1$ and $n\geq 2$ he conjectured that if the equation has finitely many solutions, then the only nontrivial solution (i.e., with $(x,y)\neq (1,1)$) is given by the case $(m,n,x,y)=(2,2,24,70)$. Jacobson, Pint\'{e}r and Walsh \cite{JAC} verified the conjecture for $n=2$ and even $m\leq 58$. Bennett, Gy\H{o}ry and Pint\'{e}r \cite{BEN} have proved it for $m\leq 11$ and arbitrary $n$. 

Also related to power sums we mention the {\em Erd\H{o}s-Moser equation}, which is the Diophantine equation
\begin{equation}
\label{eq}
S_m(k)=(k+1)^m.
\end{equation}
In a 1950 letter to Moser, Erd\H{o}s conjectured that solutions to this equation do not exist, except for the trivial solution $1^1+2^1=3^1$. Three years later, Moser \cite{MOS} proved the conjecture for odd $k$ or $m < 10^{10^6}$. Since then, much work on the Erd\H{o}s-Moser equation has been done, but it has not even been proved that there are only finitely many solutions. For surveys of work on this and related problems, see \cite{BUT,MOR} and \cite[Section D7]{Guy}.

Recently, Sondow and MacMillan studied modular versions of the Erd\H{o}s-Moser equation (\ref{eq}), in particular, the congruences
\begin{equation}\label{equ}S_m(k)\equiv (k+1)^m \pmod{k}\end{equation}
and
\begin{equation*}S_m(k)\equiv (k+1)^m\pmod{k^2}.\end{equation*}
Among other results, they proved the following. (Here and throughout the paper, $p$ denotes a prime.)
\begin{teor}[Sondow and MacMillan \cite{SON} ] \label{thmMS}
The congruence \eqref{equ} holds if and only if $p\mid k$ implies $m\equiv 0\pmod{p-1}$ and $\frac{k}{p}+1\equiv 0 \pmod{p}$. In that case, $k$ is square-free, and if $n$ is odd, then $k = 1$ or $2$.
\end{teor}

In the present paper we are interested in power sums of the form $S_m(m)$. Observe that if we put $m=k$ in equation \eqref{equ} we obtain $S_m(m)\equiv 1\pmod{m}$ where, obviously, $1$ divides $m$. This observation leads to the main goal of this paper, namely, the study of the congruence
\begin{equation}
\label{eqS_m(m)}
S_m(m)\equiv n\pmod{m},\ {\rm with}\ n\mid m.
\end{equation}
%with $n$ a divisor of $m$.

In particular we prove that if \eqref{eqS_m(m)} holds, then the quotient $Q:=m/n$ satisfies the congruence $$\sum_{p\mid Q} \frac{Q}{p} +1\equiv 0\pmod{Q}.$$
There are nine known positive integers that satisfy this congruence:
%\begin{align*}
%Q  =\ &1, 2, 6, 42, 1806, 47058, 2214502422, 52495396602,\\
%	&8490421583559688410706771261086.
%\end{align*}
$Q=1, 2,$ $6,42, 1806,$ $47058,$ $2214502422,$ $52495396602,8490421583559688410706771261086.$ We show that for each of these values of $Q$, the set of solutions $n$ to$$S_{Q n}(Q n)\equiv n\pmod{Q n}$$has positive asymptotic density strictly less than 1, except for $Q=52495396602$ when the set of solutions is empty.
%, and $Q=8490421583559688410706771261086$ when the density lies in the interval $[0,10^{-30}]$.

For instance, for $Q=1$ (i.e., if $m=n$), the set of solutions to the congruence $S_n(n)\equiv n\equiv 0\pmod{n}$ is precisely the set of odd positive integers (as proved in \cite{GMO}), whose asymptotic density is $1/2$. For $Q=2$ (i.e., if $m=2n$), the set of solutions to $S_{2n}(2n)\equiv n\pmod{2n}$
is $\{1, 2, 4, 5, 7, 8, 11, 13, 14, 16, 17, 19, 22, 23, 25, 26, 28,\dotso\}$ (see \cite[Sequence A229303]{OEIS}).

It is interesting to observe that Bernoulli's formula (\ref{int})
% and the fact that
%$$S_m(m)= \frac{1}{m+1}\big(B_{m+1}(m+1)-B_{m+1}\big)$$
would allow us to restate our results in terms of Bernoulli numbers. Nevertheless, we will not do so.

% SECTION 1
\section{The congruence $S_m(m) \equiv 1 \pmod{m}$}\label{SEC:SPE}

Let us define a set $\mathfrak{S}$ in the following way:
$$\mathfrak{S}:=\left\{m \in \mathbb{N}: S_m(m) \equiv 1 \pmod{m}\right\}.$$
The main goal of this section is to characterize the set $\mathfrak{S}$. In particular, we will see that it consists of just five elements.

We first prove two lemmas.

\begin{lem}
\label{LEM:primes}
Let $\mathcal{P}$ be a non-empty set of primes $p$ such that
\begin{itemize}
\item[i)] $p-1$ is square-free, and
\item[ii)] if $q$ is a prime divisor of $p-1$, then $q\in\mathcal{P}$.
\end{itemize}
Then $\mathcal{P}$ is one of the sets $\{2\},\{2,3\},\{2,3,7\},$ or $\{2,3,7,43\}.$
\end{lem}

\begin{proof}
Since $\mathcal{P}$ is non-empty, condition ii) implies that $2\in\mathcal{P}$.

If there exists an odd prime $p \in \mathcal{P}$, then we can define a finite sequence of primes recursively by
$$s_1=p,$$
$$s_{j}=\max\{\textrm{prime }q: q \mid s_{j-1}-1\}.$$
This sequence is contained in $\mathcal{P}$ and, since it is strictly decreasing, there exists $j$ such that $s_j=2$.
Then $s_{j-1}=2h+1$ and since 2 is the biggest prime dividing $s_{j-1}-1$, condition i) implies $h=1$ and hence $s_{j-1}=3$. If $3<p$, then in the same way we get that $s_{j-2}=3l+1$ with $l=1$ or $2$, but since $s_{j-2}$ is prime, $l=2$ and $s_{j-2}=7$. If $7<p$, then one step further leads us to $s_{j-3}=43$.

Now, if $43<p$, then $s_{j-4}=43t+1$ with $t<43$. But the only primes in $\{43t+1:t<43\}$ are $173, 431, 947, 1033, 1291, 1549, 1721$, and $173, 1033, 1549, 1721$ do not satisfy condition i), while $431, 947, 1291$ do not satisfy condition ii). Thus, none of them belongs to $\mathcal{P}$, a contradiction. Hence $43=p$. Therefore, in all cases the sequence $\{s_j\}$ has at most three elements and the only possible values for $s_1=p$ are $3,7,43$. This proves the lemma.
\end{proof}

\begin{rem}
Note that $2, 3, 7$, and $43$ are precisely the primes in \cite[Sequence A227007]{OEIS}.
\end{rem}

\begin{lem}
\label{LEM:N}
Let $\mathcal{N}$ be a set of positive integers $\nu$ such that
\begin{itemize}
\item[i)] $\nu$ is square-free, and
\item[ii)] if $p$ is a prime divisor of $\nu$, then $p-1$ divides $\nu$.
\end{itemize}
Then $\mathcal{N}\subseteq\{1,2,6,42,1806\}.$
\end{lem}
\begin{proof}
Let us define the set
$$\mathcal{S}:=\{\textrm{prime }p: p \mid \nu\ \textrm{for some $\nu \in \mathcal{N}$}\},$$
If $\mathcal{S}\neq\emptyset$, consider $p\in\mathcal{S}$. By condition ii) we have that $p-1\mid\nu$ for some $\nu\in\mathcal{N}$. By condition i) $\nu$ is square-free and since $p-1$ divides $\nu$, then $p-1$ itself is square-free. Moreover, if $q\mid p-1$ is prime, then also $q\mid \nu$. Hence, we have just seen that $\mathcal{S}$ satisfies both conditions in Lemma \ref{LEM:primes} so we conclude that $\mathcal{S}\subseteq\{2,3,7,43\}$.

Since the elements of $\mathcal{N}$ are square-free and satisfy condition ii), it is enough to proceed by direct inspection to observe that the only possible elements of $\mathcal{N}$ are $1,2,6,42,1806$, as claimed.
\end{proof}

In \cite{ERD} the set $\{2,6,42,1806\}$ was characterized as the only square-free solutions to $G(n)=\varphi(n)$, where $\varphi(n)$ is Euler's totient function and $G(n)$ stands for the number of groups (up to isomorphism) of order $n$. In \cite{KEL} it is proved that $n=1806$ is the only value for which the denominator of $B_n$ equals $n$. In \cite{DB} and \cite[Proposition 2]{ARX} several different characterizations of the set $\{1,2,6,42,1806\}$ were given. Here we present one more in terms of divisibility properties of power sums. %Namely, we have the following.

\begin{prop} \label{PROP: M}
The set $\mathfrak{S}$ consists of five elements, namely,
$$\mathfrak{S}=\{1,2,6,42,1806\}.$$
\end{prop}
\begin{proof}
Let $m\in\mathfrak{S}$, so that $S_m(m)\equiv 1\pmod{m}$. By Theorem~\ref{thmMS}, this implies that $m$ is square-free and that $p-1$ divides $m$ for every prime factor $p$ of $m$. Consequently, Lemma \ref{LEM:N} applies to obtain that $\mathfrak{S}\subseteq\{1,2,6,42,1806\}$. The proof is complete since equality can be verified computationally.
\end{proof}

\begin{rem}
The result in Proposition \ref{PROP: M} was stated without proof by Max Alekseyev in \cite[Sequence A014117]{OEIS} after the first draft of the present paper was written.
\end{rem}

We recall that an integer $n\ge2$ is called a \emph{primary pseudoperfect number} if it satisfies the equation
$$\sum_{p\mid n} \frac{n}{p}+1=n.$$
The only known primary pseudoperfect numbers are \cite{BUT}
$$2,6,42,1806,47058, 2214502422, 52495396602, 8490421583559688410706771261086.$$
It is not known whether there are infinitely many. We have just seen that the elements of $\mathfrak{S}$ are $1$ and the four primary pseudoperfect numbers $2,6,42,1806$. In the following section a family of positive integers closely related to primary pseudoperfect numbers will play a key role.

% SECTION 3
\section{The congruences $S_{Q n}(Q n)\equiv n \pmod{Q n}$}

For every $Q\in\mathbb{N}$ let us define the set
$$\mathfrak{N}_{Q}:=\{n\in\mathbb{N} :  S_{Q n}(Q n)\equiv n \pmod{Q n}\}.$$

The main goal of this section is to study the family of sets $\mathfrak{N}_{Q}$. We will make use of the following lemma \cite{GMO} several times.

\begin{lem} \label{LEM:GMO}
Let $d,k,m,$ and $t$ be positive integers.
\begin{itemize}
\item[i)] If $d$ divides $k$, then$$S_m(k)\equiv\frac{k}{d}\,S_m(d)\pmod{d}.$$
\item[ii)] Let $p^t$ be an odd prime power. Then
$$S_m(p^t)\equiv\begin{cases} -p^{t-1}\pmod{p^t}, & if\ p-1 \mid m;\\ 0\ \qquad \pmod{p^t}, & otherwise.\end{cases}$$
\item[iii)] We have
$$S_m(2^t)\equiv\begin{cases} 2^{t-1}\!\!\!\!\pmod{2^t},& \textrm{if\ $t=1$, or $t>1$ and $m>1$ is even};\\ -1\pmod{2^t}, & \textrm{if $t>1$ and $m=1$}; \\ 0\ \ \pmod{2^t}, & \textrm{if $t>1$ and $m>1$ is odd}.\end{cases}$$
\end{itemize}
\end{lem}

The first step in our study is to see that $Q$ is square-free whenever $\mathfrak{N}_{Q}\neq\emptyset$.

\begin{prop}
\label{PROP:sqfree}
If $\mathfrak{N}_{Q}$ is non-empty, then $Q$ is square-free.
\end{prop}
\begin{proof}
Fix $n\in\mathfrak{N}_{Q}$. As $1$ is square-free, we may assume that $Q>1$. Given $p\mid Q$, let $p^s$ ($s\geq 1$) be the greatest power of $p$ dividing $Q$. Let $p^r$ ($r\geq 0$) be the greatest power of $p$ dividing $n$. Since $n\in\mathfrak{N}_{Q}$, we have $S_{Q n}(Q n)\equiv n\pmod{Q n}$, and hence $S_{Q n}(Q n)\equiv n\pmod{p^{r+s}}$. Since $s\geq 1$ and $p^{r+1}\nmid n$, we get $S_{Qn}(Qn)\not\equiv 0\pmod{p^{r+s}}$. Hence, as Lemma~\ref{LEM:GMO} i) gives $S_{Q n}(Q n)\equiv \frac{Q n}{p^{r+s}}S_{Q n}(p^{r+s})\pmod{p^{r+s}}$, we get $S_{Q n}(p^{r+s})\not\equiv 0\pmod{p^{r+s}}$. Now Lemma~\ref{LEM:GMO} ii) and iii) yield $S_{Q n}(p^{r+s})\equiv \pm p^{r+s-1}\pmod{p^{r+s}}$, where the sign depends on the parity of $p$. Thus, $\pm\frac{Q n}{p^{r+s}}p^{r+s-1}\equiv n\pmod{p^{r+s}}$, and
in either case, since $\frac{Q n}{p^{r+s}}$ is coprime to $p$ and the greatest power of $p$ dividing $n$ is $p^{r}$, we get that $r+s-1=r$, so that $s=1$ as claimed.
\end{proof}

Before we present our main theorem we need to prove the following easy lemma.

\begin{lem}
\label{LEM:even}
Let $Q$ and $n$ be positive integers such that $n$ is even and $n\in\mathfrak{N}_{Q}$. Then $Q$ is also even.
\end{lem}
\begin{proof}
Assume on the contrary that $2^r$ ($r>0$) is the greatest power of $2$ dividing $n$ and that $Q$ is odd. Then $n\in\mathfrak{N}_{Q}$ and Lemma \ref{LEM:GMO} i) imply that $S_{Q n}(2^r)\equiv 0\pmod{2^r}$. But that contradicts Lemma \ref{LEM:GMO} iii).
\end{proof}

We are now in a position to characterize the pairs $(Q,n)$ such that $n\in\mathfrak{N}_Q$. In what follows we denote the set of primes dividing an integer $k$ by
$$\mathcal{P}(k) :=\{\text{prime }p: p\mid k\}.$$

\begin{teor}
\label{THEO}
Let $Q$ and $n$ be positive integers. Then $n\in\mathfrak{N}_{Q}$ if and only if the following conditions both hold.
\begin{itemize}
\item[i)] If $p\in\mathcal{P}(Q)$, then $p-1\mid Q n$ and $\frac{Q}{p}+1\equiv 0 \pmod {p}$.
\item[ii)] If $p\in\mathcal{P}(n)$ but $p\not\in\mathcal{P}(Q)$, then $p-1\nmid Q n$.
\end{itemize}
\end{teor}
\begin{proof}
Assume that $n\in\mathfrak{N}_{Q}$, so that $S_{Q n}(Q n)\equiv n \pmod{Q n}$. Then Proposition~\ref{PROP:sqfree} implies that $Q$ is square-free. Hence, we can put
$$Q=\left(\prod_{p\in\mathcal{P}(Q)\cap\mathcal{P}(n)} p\right)\left(\prod_{p\in\mathcal{P}(Q)\setminus\mathcal{P}(n)} p\right):=dQ'$$
and
$$n=\left(\prod_{p\in\mathcal{P}(n)\cap\mathcal{P}(Q)} p^{r_p}\right)\left(\prod_{p\in\mathcal{P}(n)\setminus\mathcal{P}(Q)} p^{s_p}\right):=n_1n_2.$$
Observe that $d=\gcd(Q,n)$ and $Qn=(dn_1)n_2Q'$, and that $dn_1$, $n_2$, and $Q'$ are pairwise coprime. Consequently, $S_{Q n}(Q n)\equiv n \pmod{Q n}$ holds if and only if the following three congruences all hold:
\begin{itemize}
\item[a)] $S_{Q n}(Q n)\equiv n \pmod{n_2}$,
\item[b)] $S_{Q n}(Q n)\equiv n \pmod{Q'}$,
\item[c)] $S_{Q n}(Q n)\equiv n \pmod {dn_1}$.
\end{itemize}
Let us analyze each case separately. (Note that by Lemma \ref{LEM:even} the prime $2$ cannot appear in the decomposition of $n_2$.)
\begin{itemize}
\item[a)] $S_{Q n}(Q n)\equiv n \pmod{n_2}$ if and only if $S_{Q n}(Q n)\equiv 0 \pmod{n_2}$. This happens if and only if $S_{Qn}(Q n)\equiv 0 \pmod{p^{s_p}}$ for every $p\in\mathcal{P}(n)\setminus\mathcal{P}(Q)$. By Lemma \ref{LEM:GMO} i), this happens if and only if $S_{Q n}(p^{s_p})\equiv 0 \pmod{p^{s_p}}$. Now, in \cite[Prop. 3]{GMO} it is proved that, for odd $m$, $S_k(m)\equiv 0\pmod{m}$ if and only if $q-1$ does not divide $k$ for every $q$ prime divisor of $m$. Consequently, the previous congruence holds if and only if $p-1$ does not divide $Q n$.
\item[b)] Applying Lemma \ref{LEM:GMO} i) we get that $S_{Q n}(Q n)\equiv n \pmod{Q'}$ if and only if $dS_{Qn}(Q')\equiv 1 \pmod{Q'}$, i.e., if and only if $dS_{Q n}(Q')\equiv 1 \pmod{p}$ for every $p\in\mathcal{P}(Q)\setminus\mathcal{P}(n)$. This is equivalent to $d\frac{Q'}{p}S_{Q n}(p)\equiv 1 \pmod{p}$ which, by Lemma \ref{LEM:GMO} ii), holds if and only if $p-1$ divides $Qn$ and $\frac{Q}{p}+1\equiv 0 \pmod{p}$.
\item[c)] Applying Lemma \ref{LEM:GMO} i) again and reasoning as in the previous cases, we get that $S_{Q n}(Q n)\equiv n \pmod{dn_1}$ if and only if $\frac{Q}{p}S_{Q n}(p^{r_p+1})\equiv p^{r_p} \pmod{p^{r_p+1}}$ for every $p\in\mathcal{P}(Q)\cap\mathcal{P}(n)$. Hence, $S_{Q n}(p^{r_p+1})\not\equiv 0 \pmod{p^{r_p+1}}$ and it is enough to apply Lemma \ref{LEM:GMO} ii) or iii).
\end{itemize}
To prove the converse, first observe that condition i) implies that $Q$ is square-free. Hence, we have the same decomposition $Qn=(dn_1)n_2Q'$ again. Now, since the implications in the points a), b) and c) were ``if and only if'', the result follows.
\end{proof}

In the previous section we proved that $1\in\mathfrak{N}_Q$ if and only if $Q\in\{1,2,6,42,1806\}$ and hence $Q=1$ or is a primary pseudoperfect number. We now introduce the following definition.

\begin{defi}
An integer $n\ge1$ is a \emph{weak primary pseudoperfect number} if it satisfies the congruence
$$\sum_{p\mid n}\frac{n}{p}+1\equiv 0\pmod{n}.$$
\end{defi}

In \cite[Corollary 5]{SON} it was proved that the congruence $S_m(k)\equiv 1 \pmod{k}$ holds if and only if $k$ is a weak primary pseudoperfect number and $\textrm{lcm}\{p-1 : \textrm{prime}\ p\mid k \}$ divides $m$.

Note that primary pseudoperfect numbers are trivially weak primary pseudoperfect numbers. Moreover, since the sum of the reciprocals of the first 58 primes is smaller than 2, it follows that, if it exists, a weak primary pseudoperfect number greater than $1$ which is not a primary pseudoperfect number must have at least 58 different prime factors, and so must be greater than $10^{110}$.

Theorem \ref{THEO} implies that the values of $Q$ such that $\mathfrak{N}_Q\neq\emptyset$ are weak primary pseudoperfect numbers.

\begin{cor}
\label{COR:NONEMP}
If $\mathfrak{N}_Q\neq\emptyset$, then $Q$ is a weak primary pseudoperfect number.
\end{cor}
\begin{proof}
Since $\mathfrak{N}_Q\neq\emptyset$, Theorem \ref{THEO} i) yields $\frac{Q}{p}+1\equiv 0\pmod{p}$ for every prime $p\mid Q$. Hence $Q$ is square-free and
$$\sum_{p\mid Q}\frac{Q}{p} +1\equiv \prod_{p\mid Q}\left(\frac{Q}{p}+1\right)\equiv 0\pmod{Q}$$
and the result follows.
\end{proof}

As we already pointed out, the only known weak primary pseudoperfect numbers $>1$ are the primary pseudoperfect numbers. Only eight of them are known. Hence, the only known possible values of $Q\neq 1$ that could make $\mathfrak{N}_Q\neq\emptyset$ are 2, 6, 42, 1806, 47058, 2214502422, 52495396602 and 8490421583559688410706771261086. We know from Section \ref{SEC:SPE} that $1\in\mathfrak{N}_Q$ if and only if $Q\in\{1,2,6,42,1806\}$ so, in these cases $\mathfrak{N}_Q\neq\emptyset$ and obviously $1=\min\mathfrak{N}_Q$. In the following proposition, we determine the minimal element of $\mathfrak{N}_Q$ when it exists, i.e., when $\mathfrak{N}_Q\neq\emptyset$.

\begin{prop}
\label{PROP:NQ}
Given a weak primary pseudoperfect number $Q$, define the integer
\begin{equation} \label{EQ:nQ}
\mathfrak{n}_Q:=\begin{cases}\lcm\left\{\frac{p-1}{\gcd(p-1,Q)}\ :\ p\mid Q\right\}, & \textrm{if $Q\neq 1$};\\ 1, & \textrm{if $Q=1$}.\end{cases}
\end{equation}
Then $\mathfrak{N}_Q=\emptyset$ if and only if $q-1\mid Q\mathfrak{n}_Q$ for some prime $q\mid \mathfrak{n}_Q$. Moreover, if $\mathfrak{N}_Q\neq\emptyset$, then $\mathfrak{n}_Q\mid n$ for every $n\in\mathfrak{N}_Q$ and, in particular, $\mathfrak{n}_Q=\min\mathfrak{N}_Q$.
\end{prop}
\begin{proof}
Clearly $p-1 \mid Q\mathfrak{n}_Q$ for every prime $p\mid Q$. Moreover, Theorem \ref{THEO} i) implies that if $n\in\mathfrak{N}_Q$, then $\mathfrak{n}_Q\mid n$. Applying Theorem \ref{THEO} ii) completes the proof.
\end{proof}

With this proposition we can analyze the four remaining values of $Q$ for which $\mathfrak{N}_Q$ could be non-empty.

\begin{prop}
%\begin{itemize}
%\item[i)] 
{\rm i)} If $Q\in\{47058, 2214502422, 8490421583559688410706771261086\}$, then $\mathfrak{N}_Q$ is non-empty.\\
%\item[ii)] 
{\rm ii)} The set $\mathfrak{N}_{52495396602}$ is empty.
%\end{itemize}
\end{prop}
\begin{proof}
%\begin{itemize}
%\item[i)]
{\rm i)} Using Proposition \ref{PROP:NQ}, since $47058=2\times 3\times 11\times 23\times 31$, it can be computed that $$\mathfrak{n}_{47058}=\lcm (1,1,5,1,5)=5.$$ 
In the same way we get that $\mathfrak{n}_{2214502422}=5$ and
\begin{align*}
\mathfrak{n}_{8490421583559688410706771261086}&=5\times100788283\times78595501069\\
&=39607528021345872635.
\end{align*}
To conclude it is enough to apply Proposition \ref{PROP:NQ}.\\
%\item[ii)] 
{\rm ii)} First of all, note that $Q=52495396602=2\times 3\times 11\times 17\times 101\times 149\times 3109$. By definition, both $5=\dfrac{10}{\gcd(10,52495396602)}$ and $8=\dfrac{16}{\gcd(16,52495396602)}$ divide $\mathfrak{n}_{52495396602}$. Hence, $5\mid \mathfrak{n}_{52495396602}$ and $4\mid 52495396602\times \mathfrak{n}_{52495396602}$, so that the second part of Proposition \ref{PROP:NQ} implies that $\mathfrak{N}_{52495396602}=\emptyset$ as claimed.\qedhere
%\end{itemize}
\end{proof}

It is quite surprising that the set $\mathfrak{N}_{52495396602}$ is empty. This fact implies that the converse of Corollary \ref{COR:NONEMP} is false. Hence we only know eight values of $Q$ for which $\mathfrak{N}_Q\neq\emptyset$. The following result summarizes this information.

\begin{prop}
The only known values of $Q$ for which $\mathfrak{N}_Q\neq\emptyset$ are $1, 2, 6, 42, 1806,$ $47058, 2214502422,$ and $8490421583559688410706771261086.$
\end{prop} 

Next section deals with the asymptotic density of these sets.

% SECTION 4
\section{About the asymptotic density of $\mathfrak{N}_{Q}$}

In this section we focus on the known cases when $\mathfrak{N}_{Q}$ is non-empty. In particular we are interested in studying their asymptotic density, $\delta(\mathfrak{N}_Q)$, which we show exists. For instance, the case $Q=1$ was studied in \cite[Theorem 1]{GMO}, where it was proved that $\mathfrak{N}_1$ is the set of odd positive numbers and hence has asymptotic density $1/2$. 

Here, we study the cases when $Q$ is a weak primary pseudoperfect number and the previous fact will appear as a particular case. Since the elements of $\mathfrak{N}_Q$ are always multiples of $\mathfrak{n}_Q$, a description of the complement $\mathfrak{n}_Q\mathbb{N}\setminus \mathfrak{N}_{Q}$ will be useful. Recall that {\em $p$ denotes a prime number.}

\begin{prop} \label{PROP:complement}
Let $Q$ be a weak primary pseudoperfect number such that $\mathfrak{N}_Q\neq\emptyset$. Then
$$\mathfrak{n}_Q\mathbb{N}\setminus\mathfrak{N}_Q=\bigcup_{d\mid Q} W_d(Q),$$
where the sets $W_d(Q)$ are given by
$$W_d(Q):=\left\{Kp\frac{\mathfrak{n}_Q}{D}\frac{p-1}{d}:p\nmid Q,\ d \mid p-1,\ D=\gcd\left(\mathfrak{n}_Q,\frac{p(p-1)}{d}\right),\ K\in\mathbb{N}\right\}.$$
\end{prop}
\begin{proof}
Let $n\in\mathbb{N}$ such that $n\not\in\mathfrak{N}_Q$. Then, due to Theorem \ref{THEO} ii), there exists a prime $p\nmid Q$ such that $p\mid n$ and $p-1\mid Qn$. This implies that $Qn=Ap(p-1)$ for some $A\in\mathbb{N}$. Hence, if we put $d=\gcd(Q,p-1)$ we have that $n=B\frac{p(p-1)}{d}$, where $B=dA/Q$ is an integer because $p\nmid Q$. Finally, since we want $n$ to be a multiple of $\mathfrak{n}_Q$ it follows that $B=K\frac{\mathfrak{n}_Q}{D}$ as claimed.
\end{proof}

When $\mathfrak{n}_Q=1$, then $D=1$ and Proposition~\ref{PROP:complement} can be particularized in the following way.

\begin{cor}
\label{COR:cor1}
Let $Q\in\{1,2,6,42,1806\}$. Then
$$\mathbb{N}\setminus\mathfrak{N}_{Q}=\bigcup_{d\mid Q} W_d(Q),$$
where
$$W_d(Q)=\left\{K\frac{p(p-1)}{d}:p\nmid Q,\ d \mid p-1,\ K\in\mathbb{N}\right\}.$$
\end{cor}

In the cases when $\mathfrak{n}_Q=5$, i.e., when $Q=47058$ or $2214502422$, we can also give a somewhat simpler version of Proposition 5. Note that, in these cases, $\gcd\left(\mathfrak{n}_Q,\frac{p(p-1)}{d}\right)=\gcd(\mathfrak{n}_Q,p(p-1)=1$ or 5 because $5\nmid Q$.

\begin{cor}
\label{COR:UNI}
Let $Q\in\{ 47058, 2214502422\}$. Then
$$5\mathbb{N}\setminus \mathfrak{N}_{Q}=\bigcup_{d\mid Q} W^{(1)}_d(Q)\cup \bigcup_{d\mid Q} W^{(2)}_d(Q),$$
where the sets $W_d^{(i)}(Q)$ are given by
$$W^{(1)}_d(Q):=\left\{K\frac{p(p-1)}{d}:p\nmid Q,\ 5\mid p(p-1),\ d \mid p-1,\ K\in\mathbb{N}\right\},$$
$$W^{(2)}_d(Q):=\left\{5K\frac{p(p-1)}{d}: p\nmid Q,\ 5\nmid p(p-1),\ d \mid p-1,\ K\in\mathbb{N}\right\}.$$
\end{cor}

The remaining value, $Q=8490421583559688410706771261086$, does not admit such a simple decomposition because in this case $\mathfrak{n}_Q$ is not prime. In any case, we are in a position to conclude the paper by giving bounds for the asymptotic density of $\mathfrak{N}_Q$ for the known non-empty cases. But first we introduce a technical lemma.

\begin{lem}
\label{LEM:DENS}
Let $\mathcal{A}:=\{a_k\}_{k\in\mathbb{N}}$ and $\{c_k\}_{k\in\mathbb{N}}$ be two sequences of positive integers, and for $k\in\mathbb{N}$ define the arithmetic progression
$$\mathcal{B}_{k}:=\{a_k+ (s-1)c_k: s \in \mathbb{N} \}.$$ 
If $\sum_{k=1}^{\infty}c_k^{-1}$ is convergent and $\mathcal{A}$ has zero asymptotic density, then $\bigcup_{k=1}^{\infty}\mathcal{B}_{k}$ has an asymptotic density.
\end{lem}
\begin{proof}
Let us denote $B_{n}:=\bigcup_{k=n+1}^{\infty}\mathcal{B}_{k}$ and $\vartheta(n,N):=\textrm{card}([0,N]\cap B_n)$. Then
$$\vartheta(n,N) \leq \textrm{card}([0,N]\cap \mathcal{A})+ N \sum_{k=n+1}^{\infty}\frac{1}{c_{k}}.$$
From this, we get
$$\bar{\delta}(B_{n})=\lim\sup\frac{\vartheta(n,N)}{N}\leq \lim \sup\frac{ \textrm{card}([0,N]\cap \mathcal{A})}{N} + \sum_{k=n+1}^{\infty}\frac{1}{c_{k}}= \sum_{k=n+1}^{\infty}\frac{1}{c_{k}}.$$

Now, for every $n$, the set $\bigcup_{k=1}^{n}\mathcal{B}_{k}$
has an asymptotic density $\delta_{n}$, and we have
\begin{align}
\label{bound}\delta_{n} & \leq \underline{\delta} \left(\bigcup
_{k=1}^{\infty}\mathcal{B}_{k}\right)\leq \overline{\delta} \left(\bigcup_{k=1}^{\infty}
\mathcal{B}_{k}\right)=\overline{\delta}\left(\bigcup_{k=1}^{n}\mathcal{B}_{k}\cup B_{n}\right)
 \leq \delta_{n}+\bar{\delta}(B_{n}) \\ &\leq
\delta_{n}+\sum_{k=n+1}^{\infty}\frac{1}{c_{k}}.\nonumber
\end{align}
Now, the sequence $\delta_{n}$ is non-decreasing, bounded (by 1) and, hence, convergent.
Moreover, since $\sum_{k=1}^{\infty}c_k^{-1}$ is convergent we have that $\sum_{k=n+1}^{\infty}c_k^{-1}$ converges to zero as $n\to\infty$. Taking these facts into account, 
if we take limits in (\ref{bound}) we obtain that
$$\underline{\delta} \left(\bigcup_{k=1}^{\infty}\mathcal{B}_{k}\right)=\lim_{n\to\infty}\delta_n= \overline{\delta} \left(\bigcup_{k=1}^{\infty}\mathcal{B}_{k}\right)$$
and hence $\bigcup_{k=1}^{\infty}\mathcal{B}_{k}$ has an asymptotic density, as claimed.
\end{proof}

\begin{rem} \label{Remark union}
Note that Lemma \ref{LEM:DENS} implies that if a set $S$ is the union of arithmetic progressions of the form $\{n\,c_k : n\in\mathbb{N}\}$ such that the series $\sum_{k=1}^{\infty}c_k^{-1}$ converges, then $S$ has an asymptotic density.
\end{rem}

\begin{prop}
\label{PROP:BOUND}
For every weak primary pseudoperfect number $Q$, the set $\mathfrak{N}_Q $ has an asymptotic density $\delta(\mathfrak{N}_Q)$. Moreover, $ \delta(\mathfrak{N}_Q )$ is strictly smaller than $1/\mathfrak{n}_Q$.
\end{prop}
\begin{proof} 
If $\mathfrak{N}_Q=\emptyset$, the result is clear from definition \eqref{EQ:nQ}.

On the other hand, if $\mathfrak{N}_Q \neq \emptyset$, then  Proposition \ref{PROP:complement} shows that $\mathfrak{n}_Q \mathbb{N} \setminus \mathfrak{N}_Q $ 
is the union of arithmetic progressions with difference $$\mathfrak{D}(Q,d,p):=\frac{\mathfrak{n}_Q}{\gcd(\mathfrak{n}_Q,p(p-1)/d)}\frac{p(p-1)}{d},$$ where $p$ is a prime not dividing $Q$ and $d$ is a divisor of $Q$ such that $d\mid p-1$.

Taking into account that the series $\sum_{p} \frac{1}{p(p-1)}$ is convergent, it is clear that
$$\sum_{\substack{p\textrm{ prime }\\ d \mid Q}} \frac{1}{\mathfrak{D}(Q,d,p)}<\infty.$$
Consequently, by Remark \ref{Remark union}, it follows from Lemma \ref{LEM:DENS} (with $\mathcal{A} = \{0\}$) that $\mathfrak{n}_Q \mathbb{N} \setminus \mathfrak{N}_Q $ has an asymptotic density and, hence, so does $\mathfrak{N}_Q $. Moreover, since $\mathfrak{N}_Q \subset \mathfrak{n}_Q \mathbb{N} $, it follows that $ \delta(\mathfrak{N}_Q ) < 1/\mathfrak{n}_Q$ as claimed.
\end{proof}

Proposition \ref{PROP:BOUND} not only shows that, for every weak primary pseudoperfect number $Q$, the set $\mathfrak{N}_Q $ has an asymptotic density, but also gives an upper bound for $ \delta(\mathfrak{N}_Q)$. The following result shows that $ \delta(\mathfrak{N}_Q)> 0$ if $\mathfrak{N}_Q\neq\emptyset$.

\begin{teor}
\label{TEOR:POS} 
For every weak primary pseudoperfect number $Q$ such that $\mathfrak{N}_Q$ is non-empty, the asymptotic density of $\mathfrak{N}_Q$ is strictly positive.
\end{teor} 
\begin{proof}
Fix a weak primary pseudoperfect number $Q$ such that $\mathfrak{N}_Q\neq\emptyset$. By Proposition \ref{PROP:NQ}, if we consider
$$\mathfrak{n}_Q=\min\mathfrak{N}_Q,$$
then it follows that $\mathfrak{n}_Q$ divides every element in $\mathfrak{N}_Q$.

Let $y$ be a positive integer and consider the set
$$\mathcal{M}_y:=\{ \mathfrak{n}_Qm : \textrm{ prime } p\mid m \implies p<y\}.$$
We want to study the asymptotic behavior of $[0,N]\cap\mathcal{M}_y\cap\mathfrak{N}_Q$ when $N\to \infty$.

First, we observe that
\begin{equation}\label{p1}\textrm{card}([0,N]\cap \mathcal{M}_y) = \frac{N}{Q\mathfrak{n}_Q}(\rho(y)+o(1)) \text{ as }N\to\infty,\end{equation}
where $\rho(y):=\prod_{p\leq y}(1-1/p)$. Moreover, \cite[Formula 3.25]{ROS} states that if $y\geq 285$, then
$$\rho(y)>\frac{e^{-\gamma}}{\log y}\left(1-\frac{1}{2\log^2 y}\right).$$
Hence, if we choose $y>285$ and taking into account that $e^{-\gamma}>0.56$ we obtain that 
\begin{equation}\label{p2}\rho(y)>\frac{0.5}{\log y}.\end{equation}

Now, assume that $y\geq Q$ and let $N\geq\mathfrak{n}_Qm\in\mathcal{M}_y$ is such that $\mathfrak{n}_Qm\not\in\mathfrak{N}_Q$. If this is the case, condition ii) in Theorem \ref{THEO} must fail (because condition i) trivially holds due to the form of $\mathfrak{n}_Qm$). This means that, if prime $p\mid m$, then $p-1\mid Q\mathfrak{n}_Qm$ (because the primes dividing $m$ are greater than $y>Q$). If we put $p-1=dt$ with $d\mid Q\mathfrak{n}_Q$ and $t\mid m$, then $m$ is divisible both by $t$ and $dt+1$, and therefore by their product (because they are coprime) and $t>y/d$. For fixed $d$ and $t$, the number of $m\leq N/Q\mathfrak{n}_Q$ divisible by $t(dt+1)$ is
$$\left\lfloor\frac{N}{Q\mathfrak{n}_Qt(dt+1)}\right\rfloor\leq \frac{N}{dQ\mathfrak{n}_Q t^2}$$
and we want to sum all of them over $d\mid Q\mathfrak{n}_Q$ and $t>y/d$.

If we keep $d$ fixed and sum over all $t>y/d$ and we further assume that $y\geq 2Q\mathfrak{n}_Q$ (and hence, $y/d\geq 2$) we get that
$$\sum_{t>y/d} \frac{x}{dQ\mathfrak{n}_Qt^2}< \frac{2N}{yQ\mathfrak{n}_Q}.$$

Consequently, if $\tau$ denotes the number-of-divisors function, then
\begin{equation}\label{p3}\sum_{\substack{d\mid Q\mathfrak{n}_Q\\ t>y/d}} \frac{N}{dQ\mathfrak{n}_Qt^2}\leq \frac{2\tau(Q\mathfrak{n}_Q)N}{Q\mathfrak{n}_Qy}.\end{equation}

Now, if we take $y>\max\{285,2Q\mathfrak{n}_Q\}$, putting together (\ref{p1}), (\ref{p2}) and (\ref{p3}) we obtain that, if $N\to\infty$,
\begin{equation}\label{fin}\textrm{card}([0,N]\cap \mathcal{M}_y\cap\mathfrak{N}_Q)\geq \frac{1}{Q\mathfrak{n}_Q}\left(\frac{0.5}{\log y}-\frac{2\tau(Q\mathfrak{n}_Q)}{y}\right)N(1+o(1)).\end{equation}

Finally, since $\log y$ grows more slowly than $y$, we can choose $y>\max\{285,2Q\mathfrak{n}_Q\}$ such that the main term in (\ref{fin}) is positive. This means that the asymptotic density of $\mathfrak{N}_Q$ (which exists due to Proposition \ref{PROP:BOUND}) is positive, as claimed.
\end{proof}

Observe that Proposition \ref{PROP:BOUND} and Theorem \ref{TEOR:POS} give upper and lower bounds for the asymptotic density of $\mathfrak{N}_Q$ when $\mathfrak{N}_Q$ is non-empty. The remaining results are devoted to giving tighter bounds.

\begin{prop}
\label{PROP:DENS}
We have the inequalities
$$0.0560465<\delta(\mathfrak{N}_{47058})<0.0800567,$$
$$0.0070565<\delta(\mathfrak{N}_{2214502422})<0.0800567.$$
\end{prop}

\begin{proof}

For $Q\in \{47058, 2214502422\}$ Corollary \ref{COR:UNI} implies that
$$5 \mathbb{N}\setminus \mathfrak{N}_{Q}=\bigcup_{\substack{d\mid Q\\ p\textrm{ prime}}} \mathcal{U}_{d,p}(Q),$$
where
$$\mathcal{U}_{d,p}(Q):=\begin{cases}

    \left\{5K\frac{p(p-1)}{d}:\ K\in\mathbb{N}\right\}, & \textrm{ if }p\nmid Q,\ 5\nmid p(p-1),\ d \mid p-1;\\

   \left\{K\frac{p(p-1)}{d}: K \in\mathbb{N}\right\}, & \textrm{ if } p\nmid Q,\ 5\mid p(p-1),\ d \mid p-1;\ \\

    \emptyset, & \textrm{ otherwise.}\end{cases}$$

If $p_i$ denotes the $i$th prime number, we have that
 \begin{equation}\label{eq1}\delta(\bigcup_{\substack{i \leq 50   \\d\mid Q }}\mathcal{U}_{d,p_i}(Q))< \delta(5 \mathbb{N}\setminus \mathfrak{N}_{Q})\leq \delta(\bigcup_{\substack{i \leq 50   \\d\mid Q }}\mathcal{U}_{d,p_i}(Q))+\sum_{\substack{i > 50   \\d\mid Q}} \delta(\mathcal{U}_{d,p_i}(Q)).\end{equation}

Computing the primitive elements of the union of $\mathcal{U}_{d,p_i}(Q)$ for  $i\leq50$ we obtain the set
%$$\begin{array}{ll} \mathcal{S}_{50}:=\{ 10,35,235,285,335,695,2985,3775,5135,8515,8555, 8755,17015,18145, \\ \hspace{16ex} 22005,28355, 41255,69305,79655,128255\}\end{array}$$
\begin{align*} \mathcal{S}_{50}:=\{ &10,35,235,285,335,695,2985,3775,5135,8515,8555, 8755,17015,18145,\\
 &22005,28355, 41255,69305,79655,128255\}
 \end{align*}
and consequently
    $$ \bigcup_{\substack{i \leq 50   \\d\mid Q }}\mathcal{U}_{d,p_i}(Q)= \bigcup_{\substack{t \in \mathcal{S}_{50}}} \{K t : K \in \mathbb{N} \}.$$

Now, using Mathematica, the inclusion-exclusion principle leads to
%$$0.119781<\delta(\bigcup_{\substack{i \leq 50   \\d\mid Q }}\mathcal{U}_{d,p_i}(Q))=\frac{22232316750436016599664077073}{185607973526408781310806785365}$$
$$\delta(\bigcup_{\substack{i \leq 50   \\d\mid Q }}\mathcal{U}_{d,p_i}(Q))=\frac{48357225625417447595522734010896225250266313}{403167008827681283131141033075588326251331565}.$$

Using now $\sum_{\substack{p}} \frac{1}{p(p-1)}=0.7731566690497\dotso$ and $\delta(\mathcal{U}_{d,p_i}(Q))\leq \frac{d}{p_i (p_i-1)}$, and doing some computations, we can also obtain an upper bound for $\displaystyle{\sum_{\substack{i > 50   \\d\mid Q}} \delta(\mathcal{U}_{d,p_i}(Q))}$ which, together with (\ref{eq1}), leads to the desired bounds.
\end{proof}

\begin{prop}
We also have the inequalities
$$0.583874 < \delta(\mathfrak{N}_2) < 0.584604,$$
$$0.70405 < \delta(\mathfrak{N}_6) < 0.707659,$$
$$0.78215 < \delta(\mathfrak{N}_{42}) < 0.79399,$$
$$0.7747 < \delta(\mathfrak{N}_{1806}) < 0.812570.$$
\end{prop}
\begin{proof} 
For every $Q\in \{2,6,42,1806\}$, using Corollary \ref{COR:cor1} we have that
$$\mathbb{N}\setminus\mathfrak{N}_{Q}=\bigcup_{\substack{d\mid Q\\ p\textrm{ prime}}} \mathcal{U}_{d,p}(Q),$$
where
$$\mathcal{U}_{d,p}(Q):=\begin{cases}

    \left\{K\frac{p(p-1)}{d}:\:\ K\in\mathbb{N}\right\}, & \textrm{if $p\nmid Q$, $d \mid p-1$};\\

    \emptyset, & \textrm{otherwise}.\end{cases}$$
Then, the required computations are similar to those in the previous proof and can be performed with Mathematica using again the inclusion-exclusion principle.
\end{proof}

In the case of $Q_9:=8490421583559688410706771261086$ it does not seem feasible to apply the ideas and techniques from Proposition~\ref{PROP:DENS}, because the computations would require an unaffordable amount of time. Nevertheless, we can improve the bounds $0<\delta(\mathfrak{N}_{Q_9})<1$. Namely, this density by Proposition~\ref{PROP:BOUND} is less than $1/\mathfrak{n}_{Q_9}$, yielding $\delta(\mathfrak{N}_{Q_9})<10^{-30}$, and taking $y=Q_9\mathfrak{n}_{Q_9}$ in the proof of Theorem~\ref{TEOR:POS} gives $\delta(\mathfrak{N}_{Q_9})>1.2\times 10^{-53}$ from (\ref{fin}). To provide tighter bounds remains a possibility. 

The following table summarizes the main results about the known values of $Q$ for which $\mathfrak{N}_Q$ is non-empty.

\vspace{-1em}
\begin{center}
\begin{table}[ht]
\caption{Known weak primary pseudoperfect numbers $Q$ if $\mathfrak{N}_Q\neq\emptyset$}
\vspace{-1em}
\begin{tabular}{|c|c|c|}
\hline $Q$ & $\min\mathfrak{N}_Q$ & Bounds for $\delta(\mathfrak{N}_Q)$\\ \hline  1 & 1 & $1/2,1/2$\\ \hline  2 & 1 & $0.583874,0.584604$\\ \hline 6 & 1 & $0.70405,0.707659$\\ \hline 42 & 1 & $0.78215,0.79399$\\ \hline 1806 & 1 & $0.7747,0.812570$\\ \hline 47058 & 5 & $0.0560465,0.080057$\\ \hline  2214502422 & 5 & $0.0070565,0.080057$\\ \hline 8490421583559688410706771261086 & 39607528021345872635 & $1.2\times 10^{-53},10^{-30}$\\ \hline
\end{tabular}
\end{table} 
\end{center}
\vspace{-2em}

Observe that we have
$$1/2=\delta(\mathfrak{N}_1)<\delta(\mathfrak{N}_2)<\delta(\mathfrak{N}_6)<\delta(\mathfrak{N}_{42})<1.$$
However, we do not know whether the relation $\delta(\mathfrak{N}_{42})< \delta(\mathfrak{N}_{1806})$ holds.

In fact, $\delta(\mathfrak{N}_{Q})$ is not increasing with $Q$ since, for instance, $\delta(\mathfrak{N}_{47058})$ and $\delta(\mathfrak{N}_{2214502422})$ are each less than $\delta(\mathfrak{N}_1)$. On the other hand, if we observe that
\begin{equation*}
\mathfrak{n}_{1}=\mathfrak{n}_{2}=\mathfrak{n}_{6}=\mathfrak{n}_{42}=\mathfrak{n}_{1806}=1 \quad\text{and}\quad
\mathfrak{n}_{47058}=\mathfrak{n}_{2214502422}=5,
\end{equation*}
we can end the paper with the following prediction.

\begin{con}
If $Q$ and $ Q'$  are weak primary pseudoperfect numbers such that $\mathfrak{N}_{Q}$ and $\mathfrak{N}_{Q'}$ are non-empty and $\mathfrak{n}_{Q}<\mathfrak{n}_{Q'}$, then $\delta(\mathfrak{N}_{Q} ) > \delta(\mathfrak{N}_{Q'} )$.
\end{con} 

\section*{Acknowledgements}

We are grateful to the anonymous referee for suggestions on asymptotic density and, especially, for the proof of Theorem \ref{TEOR:POS}. We thank Bernd Kellner for useful comments on several references.

\end{document}